\documentclass[12pt]{amsart}
\usepackage[centertags]{amsmath}
\usepackage{latexsym}
\usepackage{amssymb}

\newcommand{\F}{\mathbb F}

\newcommand{\barFq}{\overline{{\mathbb F}_q}}
\newcommand{\Proj}{\mathbb P}

\newcommand{\hide}[1]{}
\newtheorem{dummy}{Dummy}

\numberwithin{equation}{section}

\newtheorem{theorem}[dummy]{Theorem}

\newtheorem{cor}[dummy]{Corollary}

\theoremstyle{definition}

\theoremstyle{remark}

\begin{document}

\bibliographystyle{amsalpha}
\author{Sandro Mattarei}
\email{mattarei@science.unitn.it}
\address{Dipartimento di Matematica\\
  Universit\`a degli Studi di Trento\\
  via Sommarive 14\\
  I-38123 Povo (Trento)\\
  Italy}
\title{A property of the inverse of a subspace of a finite field}
\begin{abstract}
We prove a geometric property of the set
$A^{-1}$ of inverses of the nonzero elements of an $\F_q$-subspace $A$ of a finite field
involving the size of its intersection with two-dimensional $\F_q$-subspaces.
We give some applications,
including a new upper bound on $|A^{-1}\cap B|$
when $A$ and $B$ are $\F_q$-subspaces of different dimension of a finite field,
satisfying a suitable natural assumption.
\end{abstract}
\subjclass[2000]{Primary 11T30; secondary 51E20}
\keywords{finite field, subspace, inverse, arc, cap}
\maketitle

\section{Introduction}

The inversion map $x\mapsto x^{-1}$ in a finite field has been the object of various studies in recent years.
In particular, its interaction with the operation of addition is of interest for cryptographic applications.
The best-known example is that inversion in the finite field of $2^8$ elements
({\em patched} by sending zero to itself)
is the nonlinear transformation employed in the S-boxes
in the Advanced Encryption Standard (Rijndael, see~\cite{AES}).
A study of an AES-like cryptosystem in~\cite{CDVS:AES} required, in the special case of finite fields, the determination
of the additive subgroups of a field which are closed with respect to inverting nonzero elements,
which was provided by the author in~\cite{Mat:inverse-closed}.
(The more general question in division rings was independently answered in~\cite{GGSZ}.)
A small variation of this fact was required in~\cite{KLS} for a different cryptographic application,
and a more substantial generalization was studied in~\cite{Csajbok:inverse-closed},
to which we will return later in this Introduction.
All those studies involve a set
$A^{-1}=\{x^{-1}:0\neq x\in A\}$, where $A$ is an $\F_q$-subspace of a finite field.

In this note we prove a geometric property of $A^{-1}$.
Because the specific ambient finite field plays no role in our result, it will be equivalent, but notationally simpler, to
rather work inside an algebraic closure $\barFq$ of $\F_q$.
Our result then reads as follows.

\begin{theorem}\label{thm:geometric}
Let $A$ be an $\F_q$-subspace of $\barFq$, of dimension $d$.
Let $U$ be a two-dimensional $\F_q$-subspace of $\barFq$,
and suppose
$|A^{-1}\cap U|>d(q-1)$.
Then $U^{-1}\subseteq A$, and
the $\F_q$-span of $U^{-1}$ is a one-dimensional $\F_{q^e}$-subspace of $\barFq$,
for some $e$.
\end{theorem}

Because set $A^{-1}$ is closed with respect to scalar multiplication by elements of $\F_q^\ast$,
it is natural to interpret its properties in a projective space
$\Proj V$, where $V$ is an $\F_q$-linear subspace of $\barFq$ containing $A^{-1}$.
(To resolve possible ambiguities,
in this paper the operator $\Proj$ will only be applied to vector spaces over the field $\F_q$.)
From this geometric viewpoint, the two-dimensional space $U$ of Theorem~\ref{thm:geometric}
represents a line in a projective space, and the condition on the integer $|A^{-1}\cap U|/(q-1)$, which clearly equals
$|A\cap U^{-1}|/(q-1)$, can be read in terms of caps or arcs according as whether we focus our attention on $A^{-1}$ or $U^{-1}$.
We present some applications of our result which exploit, in turn, one or the other interpretation.

Our first application of Theorem~\ref{thm:geometric} is a simpler proof of the main result of~\cite{Faina:cyclic_model}, which is their Theorem~3.3,
and generalizes
the following result of M.~Hall~\cite{Hall:difference_sets}:
in the cyclic model of the projective plane $PG(2,q)$, the inverse of a line is a conic.
We will explain this terminology, and state and prove the main result of~\cite{Faina:cyclic_model} in Section~\ref{sec:arcs},
after proving Theorem~\ref{thm:geometric}.

Our next application, which we formalize in Theorem~\ref{thm:cap}, uses Theorem~\ref{thm:geometric} to deduce an upper bound on
$|A^{-1}\cap B|/(q-1)$ from any available general bounds on (higher) caps,
where $A$ and $B$ are finite-dimensional subspaces of $\barFq$.
Theorem~\ref{thm:geometric} is the special case of this where $B$ has dimension two.
Of course a non-trivial bound can only be obtained provided one steers away from some special configurations,
such as the extreme case $A^{-1}\subseteq B$.
By employing a general bound on caps our result yields
$|A^{-1}\cap B|\le (d-1)|B|/q+q-d$, where $|A|=q^d$, under the assumption that $A$ does not contain any
any nonzero $\F_{q^e}$-subspace of $\barFq$ with $e>1$.

The special case of this bounding problem where $|A|=|B|$ was
studied in~\cite{Csajbok:inverse-closed}, and then in~\cite{Mat:inversion}.
In particular, in the former Csajb\'{o}k proved the general bound
$|A^{-1}\cap B|\le 2|B|/q-2$, for any subspaces $A$ and $B$ of $\barFq$, of the same (finite) dimension, such that $A^{-1}\not\subseteq B$.
This surpasses the bound given by our Theorem~\ref{thm:cap} when $|A|=|B|=q^d>q^3$.
However, the method of~\cite{Csajbok:inverse-closed},
which expands on a polynomial argument of the author in~\cite{Mat:inverse-closed},
seems unsuited to deal with the case $|A|>|B|$,
where our Theorem~\ref{thm:cap} provides the only known nontrivial bound.
(Note that our bound is larger than $|B|-1$ when $|A|<|B|$, and hence trivial.)

Furthermore, Theorem~\ref{thm:cap} produces a contribution to the case $|A|=|B|=q^3$,
where a slightly better available bound on caps yields
$|A^{-1}\cap B|<2|B|/q-2$
apart from a special situation.
Our final application of Theorem~\ref{thm:cap} is then the determination,
in Theorem~\ref{thm:Segre},
of an exceptional geometric configuration
which occurs when equality is attained in Csajb\'{o}k's bound for $|A|=|B|=q^3$:
the image of $A^{-1}\cap B$ in $\Proj B$ is then the union of a conic and an external line.
As we explain in Section~\ref{sec:caps},
that result is included in a more general investigation in~\cite{Mat:inversion},
but the short proof given here bypasses longer and more demanding arguments employed there.

The author is grateful to Bence Csajb\'{o}k for interesting discussions on this topic.

\section{A proof of Theorem~\ref{thm:geometric}, and an application involving arcs}\label{sec:arcs}

\begin{proof}[Proof of Theorem~\ref{thm:geometric}]
Our hypothesis means that there exist $\xi,\eta\in\barFq$
with $U=\F_q\xi+\F_q\eta$, and $d$ distinct $\alpha_1,\ldots,\alpha_d\in\F_q$, such that
$\eta,\xi+\alpha_1\eta,\ldots,\xi+\alpha_d\eta\in A^{-1}$.
The inverses of those elements must then be linearly dependent over
$\F_q$, because $A$ has dimension $d$.
Consider a shortest linear dependence relation among them.
Possibly after permuting those elements, which may include redefining $\eta$,
the relation takes the form
\[
\frac{1}{\eta}+\sum_{i=1}^e\frac{\beta_i}{\xi+\alpha_i\eta}=0,
\]
for some
$\beta_1,\ldots,\beta_e\in\F_q^\ast$, with $2\le e\le d$.
Clearing the denominators we find that the pair $(\xi,\eta)$
is a zero of a homogeneous polynomial of degree $e$ with coefficients in $\F_q$ and,
consequently, $\xi/\eta\in\F_{q^t}$ for some $t\le e$.
Because
$\xi/\eta,\xi/(\xi+\alpha_1\eta),\ldots,\xi/(\xi+\alpha_{e-1}\eta)$
belong to $\F_{q^t}$ and are linearly independent over $\F_q$,
we have $t=e$, and they form a basis of
$\F_{q^e}$ over $\F_q$.
Hence the elements
$1/\eta,1/(\xi+\alpha_1\eta),\ldots,1/(\xi+\alpha_{e-1}\eta)$
of $A\cap U^{-1}$ span the one-dimensional $\F_{q^e}$-subspace
$\F_{q^e}\xi^{-1}$ of $\barFq$.
Because $\eta$ and $\xi+\alpha_1\eta$ belong to
$(\F_{q^e}\xi^{-1})^{-1}\cup\{0\}=\F_{q^e}\xi$
and their $\F_q$-span equals $U$,
we have $U\subseteq\F_{q^e}\xi$,
and hence
$U^{-1}\subseteq\F_{q^e}\xi^{-1}\subseteq A$.
\end{proof}

In order to formulate our first application  of Theorem~\ref{thm:geometric} we need to introduce some terminology.
The {\em cyclic model of $PG(n,q)$} is the $n$-dimensional projective space $\Proj\F_{q^{n+1}}$, with the added cyclic group structure induced by
the multiplicative group $\F_{q^{n+1}}^{\ast}$.
The {\em inverse} (called the {\em additive inverse} in~\cite{Faina:cyclic_model}) of a subset of the cyclic model of
$PG(n,q)$ must be intended with respect to the group operation.
An arc in $PG(n,q)$ is a set of $k\ge n+1$ points of which no $n+1$ lie on the same hyperplane.
We use our Theorem~\ref{thm:geometric} to prove the main result of~\cite{Faina:cyclic_model},
which reads as follows.

\begin{theorem}[Theorem~3.3 of~\cite{Faina:cyclic_model}]\label{thm:Faina}
If $q+1>n$, then in the cyclic model of $PG(n,q)$ the inverse of any line is an arc in some subspace $PG(m,q)$, where $m+1$ divides $n+1$.
\end{theorem}

\begin{proof}
A line in the cyclic model of $PG(n,q)$ is the image in $\Proj\F_{q^{n+1}}$ of a two-dimensional $\F_q$-subspace $U$ of $\F_{q^{n+1}}$.
Fix such a line and let $V$ be the $\F_q$-span of $U^{-1}$.
If $V$ has dimension $m+1$ then $\Proj V$ is a projective geometry $PG(m,q)$.
Any hyperplane in $\Proj V$, which is the image of an $m$-dimensional $\F_q$-subspace $A$ of $V$,
meets the image of $U^{-1}$ in $\Proj V$ in at most $m$ points,
otherwise Theorem~\ref{thm:geometric} would be contradicted because $U^{-1}\not\subseteq A$.
Hence the image of $U^{-1}$, which is the inverse of our line, is an arc in $\Proj V$.

It remains to show that $m+1$ divides $n+1$, and to this purpose we may assume $m<n$.
Because $|V\cap U^{-1}|/(q-1)=|U^{-1}|/(q-1)=q+1>m+1$, an application of Theorem~\ref{thm:geometric}
with $V$ in place of $A$ shows that $V$
is a one-dimensional $\F_{q^e}$-subspace of $\barFq$,
whence $e=m+1$.
Because $V\subseteq\F_{q^{n+1}}$ it follows that $\F_{q^{m+1}}\subseteq\F_{q^{n+1}}$, and hence $m+1$ divides $n+1$.
\end{proof}

If $n+1$ is a prime in Theorem~\ref{thm:Faina}, one concludes with~\cite{Faina:cyclic_model} that the inverse of a line is an arc in $PG(n,q)$,
and for $n=2$ one recovers the result of M.~Hall mentioned in the introduction.

\section{Applications involving caps}\label{sec:caps}

Our next application of Theorem~\ref{thm:geometric} concerns caps rather than arcs.
A $(k,r)$-cap in the projective geometry $PG(n,q)$
is a set of $k$ points, of which no $r+1$ are collinear.
(A variant of this definition requires that the set contains at least one set of $r$
collinear points, but this difference is immaterial here.)
The largest size $k$ of a $(k,r)$-cap
in $PG(n,q)$ is denoted by $m_r(n,q)$.

\begin{theorem}\label{thm:cap}
Let $A$ and $B$ be $\F_q$-subspaces of $\barFq$ of size $q^d$ and $q^{d'}$, respectively.
Suppose that $A$ does not
contain any nonzero $\F_{q^e}$-subspace of $\barFq$ with $e>1$.
Then
$|A^{-1}\cap B|/(q-1)\le m_d(d'-1,q)$.
\end{theorem}

\begin{proof}
Suppose for a contradiction that the desired conclusion is violated, that is,
$|A^{-1}\cap B|/(q-1)=k>m_d(d'-1,q)$.
Then the image of $A^{-1}\cap B$ in $\Proj B\cong PG(d'-1,q)$ is not a $(k,d)$-cap, and hence
it meets a line in $\Proj B$ in more than $d$ points.
According to Theorem~\ref{thm:geometric}, the preimage $U$ in $B$ of that line is contained in some
one-dimensional $\F_{q^e}$-subspace $\F_{q^e}\xi$ of $\barFq$ with $e>1$, which in turn is contained in $A^{-1}\cup\{0\}$.
But then $A$ contains $\F_{q^e}\xi^{-1}$, contradicting our hypotheses.
\end{proof}

When $d>3$ the only general bound on cap sizes
available for use in Theorem~\ref{thm:geometric} is
$m_r(t,q)\le 1+(r-1)\cdot(q^t-1)/(q-1)$,
which is easily proved by considering all lines which pass through a fixed point of the cap.
The conclusion of Theorem~\ref{thm:geometric} then reads
$|A^{-1}\cap B|\le (d-1)q^{d'-1}+q-d$.
As we noted in the Introduction, this is nontrivial when $d\ge d'$, and new when $d>d'$,
whereas neither the polynomial method used in~\cite{Csajbok:inverse-closed}, nor its more powerful variant employed in~\cite{Mat:inversion},
seem capable to produce any essentially nontrivial bound in the latter case.
When $d'=3$ our bound can perhaps be more conveniently written as
$|A^{-1}\cap B|/(q-1)\le(d-1)q+d$.

When $d=d'>2$ the bound we have just given is worse than the bound
$|A^{-1}\cap B|\le 2q^{d-1}-2$
proved by Csajb\'{o}k in~\cite{Csajbok:inverse-closed}.
(They match when $d=d'=2$, an easy case briefly discussed in~\cite[Section~2]{Mat:inversion}.)
For $d>3$ the author strengthened Csajb\'{o}k's bound to
$|A^{-1}\cap B|\le q^{d-1}+O_d(q^{d-3/2})$
in~\cite{Mat:inversion}.
However, Csajb\'{o}k's bound is sharp when $d=3$, and the following corollary of Theorem~\ref{thm:cap}
provides crucial information on the case where equality is attained.

\begin{cor}\label{cor:q^3}
Suppose $q>3$.
Let $A$ and $B$ be $\F_q$-subspaces of $\barFq$ of size $q^3$, with $A^{-1}\not\subseteq B$.
If
$|A^{-1}\cap B|/(q-1)>2q+1$, then $(A^{-1}\cap B)\cup\{0\}$
contains a one-dimensional $\F_{q^2}$-subspace of $\barFq$.
\end{cor}

\begin{proof}
The general bound for $m_r(t,q)$ recalled above reads $m_3(2,q)\le 2q+3$ in the case of present interest.
However, the latter can be improved to
$m_3(2,q)\le 2q+1$ for $q>3$,
see~\cite[Corollary~12.11 and Theorem~12.47]{Hirschfeld}.
Consequently, under our hypothesis
the conclusion of Theorem~\ref{thm:cap} does not hold,
and hence the argument in the proof of Theorem~\ref{thm:cap} applies.
We necessarily have $e=2$,
the one-dimensional $\F_{q^2}$-subspace $\F_{q^2}\xi$ of $\barFq$
found there coincides with $U$, and hence
it is not only contained in $A^{-1}\cup\{0\}$,
but in $B$ as well.
\end{proof}

The information contained in the conclusion of Corollary~\ref{cor:q^3}, together with a further appeal to Theorem~\ref{thm:geometric}
and to a classical result of B.~Segre, is sufficient to determine the geometric structure
of the set $A^{-1}\cap B$ for three-dimensional $\F_q$-subspaces which attain equality in Csajb\'{o}k's bound, as follows.

\begin{theorem}\label{thm:Segre}
Suppose $q$ odd and $q>3$.
Let $A$ and $B$ be $\F_q$-subspaces of $\barFq$ of size $q^3$, with $A^{-1}\not\subseteq B$,
such that
$|A^{-1}\cap B|/(q-1)=2q+2$.
Then the image of $A^{-1}\cap B$ in $\Proj B$ is the union of a line
and a conic.
\end{theorem}

\begin{proof}
According to Corollary~\ref{cor:q^3}, the set $(A^{-1}\cap B)\cup\{0\}$
contains a one-dimensional $\F_{q^2}$-subspace $\F_{q^2}\xi$ of $\barFq$.
After replacing the subspaces $A$ and $B$ with $\xi A$ and $\xi^{-1}B$,
which changes neither the hypotheses nor the conclusion, we may assume that
$(A^{-1}\cap B)\cup\{0\}$ contains the subfield $\F_{q^2}$ of $\barFq$.
Thus, both $A$ and $B$ contain $\F_{q^2}$.
The image of $\F_{q^2}$ in the two-dimensional projective space $\Proj(B)$
is the required line.

Now set $C:=(A^{-1}\cap B)\setminus\F_{q^2}$.
We claim that any two-dimensional $\F_q$-subspace $U$ of $B$ meets $C$ in at most $2(q-1)$ elements.
Assuming $U\neq\F_{q^2}$ as we obviously may, we have $|U\cap\F_{q^2}|=q$.
By way of contradiction, suppose that $|C\cap U|>2(q-1)$.
Then $|A^{-1}\cap U|>3(q-1)$, and hence
$U^{-1}$ spans a one-dimensional $\F_{q^2}$-subspace of $\barFq$
according to Theorem~\ref{thm:geometric}.
This being clearly not the case, we have to concede that $|C\cap U|\le 2(q-1)$.
Thus, the image of $C$ in the two-dimensional projective space $\Proj B$
is an arc with $q+1$ points.
According to a celebrated result of B.~Segre~\cite[Theorem~8.14]{Hirschfeld},
when $q$ is odd any such arc is a conic.
\end{proof}

The very special case of Theorem~\ref{thm:Segre} where $A=B\subseteq\F_{q^4}$
was proved by Csajb\'{o}k~\cite[Theorem~4.8, Assertion~(3)]{Csajbok:inverse-closed}.
A much more general result than Theorem~\ref{thm:Segre} was proved by the author
by different methods in~\cite[Theorem~9]{Mat:inversion},
which gives a classification, and with it a precise count in a suitable sense, of all pairs of three-dimensional $\F_q$-subspaces $A$, $B$ of $\barFq$ such that
$|A^{-1}\cap B|/(q-1)=\{2q,2q+1,2q+2\}$,
with no restriction on the parity of $q$
(with $q>5$ for the two smaller values).
It turns out that in all those cases the image of $A^{-1}\cap B$ in $\Proj B$
is the union of a nonsingular conic and a secant, tangent or external line in the three cases.
The intermediate case occurs only for even $q$, and the other two cases only for odd $q$.

\bibliography{References}

\end{document}